\documentclass[a4paper,8pt]{article}
\usepackage[english]{babel}
\usepackage{geometry}
\usepackage[T1]{fontenc}
\usepackage[latin1]{inputenc}
\usepackage{amsmath,amssymb,mathrsfs,amsthm}
\usepackage{soul}
\usepackage{epsfig}
\usepackage{hyperref}
\usepackage{graphicx}
\usepackage{xypic}
\usepackage{enumerate}
\usepackage{datetime}
\usepackage{fancyhdr}
\fancypagestyle{plain}{
\lhead{}
\rhead{}
\lfoot{Date \today{} }

}

\usepackage{tocloft}

\newtheorem{theorem}{Theorem}[section]
\newtheorem{proposition}[theorem]{Proposition}
\newtheorem{claim}[theorem]{Claim}

\newtheorem{corollary}[theorem]{Corollary}

\newcommand{\C}{\mathbb{C}}

\newcommand{\Z}{\mathbb{Z}}

\newcommand{\p}{\mathbb{P}}

\newcommand{\Q}{\mathbb{Q}}
\newcommand{\N}{\mathbb{N}}

\title{On the zeta functions on the projective
complex spaces.}
\date{\today}
\author{Mounir Hajli\footnote{Institute of Mathematics, Academia Sinica, Astronomy-Mathematics Building,
6F, No.1, Sec.4, Roosevelt Road, Taipei 10617, Taiwan. 
\emph{E-mail}:\,\ttfamily{hajli@math.sinica.edu.tw, hajlimounir@gmail.com}}}

\begin{document}

\maketitle
\begin{abstract}
In this article, we study the zeta function $\zeta_q$
associated to the Laplace operator $\Delta_q$ acting
on the space of the smooth  $(0,q)$-forms with $q=0,\ldots,n$ on the complex projective space $\p^n(\C)
$ endowed with its Fubini-Study metric. In particular, we show that
the values of $\zeta_q$ at non-positive integers are
rational. Moreover, we give a formula for $
\sum_{q\geq 0}(-1)^{q+1}q\zeta_q'(0)
,$ the associated holomorphic analytic torsion. 
\end{abstract}

\begin{center}
MSC: 11M41.
\end{center}

\tableofcontents
\section{Introduction}
Let $\p^n(\C)$ be the complex projective space of
dimension $n$ endowed with its Fubini-Study metric. 
For
 $q=0,1,\ldots,n$, we denote by  $\zeta_q$  the zeta function of the Laplace 
operator $\Delta_q$ associated to the Fubini-Study metric and  acting on $A^{0,q}(\p^n(\C))$,
the space of smooth $(0,q)$-forms on $\p^n(\C)$. The spectrum
of $\Delta_q$ was computed by Ikeda and Taniguchi in 
\cite{Spectra}.  It is known that 
$\zeta_q$ is a holomorphic function on $\Re(s)>n$ 
and it admits a meromorphic extension to $\C$ which is 
holomorphic at $s=0$ and has poles in $\{1,\ldots,n\}
$. This follows from the general theory of Laplace
operators 
on compact K\"ahler manifolds, see for instance 
\cite{heat}.

The zeta function $\zeta_0$ associated to 
$\Delta_0$ associated to the Fubini-Study metric was studied in \cite{SpreaficoProj}. 

In this article, we study $\zeta_q$ for $1\leq q\leq n
$. Our first goal is to  give explicit formulas for its values
at non-positive integers. First, let us  recall the expression of $\zeta_q$. 
By the computation of
\cite{Spectra} or see  \cite[p. 33]{GSZ},   we have
for any  $1\leq q<n$, 
\begin{equation}\label{fin1}
\zeta_q(s)=\overline{\zeta}_q(s)+\overline{\zeta}_{q+1
}(s),
\end{equation}
and 
\[
\zeta_n(s)=\overline{\zeta}_n(s),
\] where
\[
\overline{\zeta}_q(s)=\sum_{k\geq q} \frac{d_{n,q}(k)}{\bigl(
k(k+n+1-q)\bigr)^s}\quad \text{for} \,q=1,\ldots,n,
\]
with  $d_{n,q}(k)=\bigl(\frac{1}{k}+\frac{1}{k+n+1-q}  \bigr)
\frac{(k+n)!(k+n-q)!}{k!(k-q)!n!(n-q)!(q-1)!}$.\\

In Theorem \eqref{ValuesZeta} we give a formula 
for $\overline{\zeta}_q(-m)$ for any $m\in \N$. The
method of the proof is a generalization of the approach
given in \cite{Weisberger}.  Roughly speaking, the idea of the proof
consists of introducing an auxiliary function 
$\xi_q$ and to write $\overline{\zeta}_q$ as 
an infinite sum involving the $\xi_q$ (see Claim \eqref{n1}).
Then, we show that  the computation of $\overline{\zeta}_q(-m)$ reduces
to the computation of the values of $\xi_q$ at 
non-positive integers, which is done in Proposition
\eqref{prop1}.

  When 
$n=1$, we have a more simple expression. Namely, if
 we denote
by $\zeta^{(0)}$ the zeta function associated to 
the Fubini-Study metric on $\p^1(\C)$, then we have
\begin{theorem}[Theorem \eqref{xx1}]
For any $m\in \N$, 
\[ 
 \zeta^{(0)}(-m)=-\frac{1}{m+1}\sum_{k=m+1}^{2m+2}(-1)^k \binom{m+1}{2m+2-k}B_k.
\]
where $B_k$ is the $k$-th Bernoulli number.\\
\end{theorem}

The second goal of this article is the study of
the holomorphic analytic torsion on $\p^n(\C)$.
Recall that the holomorphic analytic torsion associated to 
$\p^n(\C)$ endowed with its Fubini-Study metric is
by definition the following real number
\[
\sum_{q\geq 0}(-1)^{q+1}q\zeta_q'(0),
\]
also called the regularized determinant. This 
notion was first introduced by Ray and Singer in
\cite{RaySinger}. In \cite{Quillen}, Quillen 
uses the holomorphic analytic torsion in order to
define a smooth metric     on the determinant of 
cohomology of the direct image of a given
 holomorphic 
vector bundle; Roughly speaking, given 
$f:(X,\omega_X)\rightarrow (Y,\omega_Y)$ a K\"ahler fibration between K\"ahler compact manifolds and 
$\overline{E}$ a hermitian vector bundle on $X$
we can endow $\lambda(E):=\det(f_{\ast} E)$ (the 
determinant of cohomology of $f_\ast E$) with 
a smooth metric $h_Q$ given for any $y\in Y$ 
as the product of the $L^2$-metric on the fiber 
$f^{-1}(y)$ and the holomorphic analytic torsion
of the restriction $\overline{E}_y$ on $f^{-1}(y)$ 
endowed with $\omega_y:=\omega_{X|_{f^{-1}(y)}}$ 
(see also \cite{BGS1}, \cite{BGS2} and \cite{BGS3}).\\

The holomorphic analytic torsion enters into the 
formulation of the Arithmetic Riemann Roch theorem
of Gillet and Soul\'e (see \cite{GSZ} and 
\cite{ARR}). In order to prove this theorem, one needs
to compute the holomorphic analytic torsion of 
$\p^n(\C)$ for the Fubini-Study metric. This is was done
in  \cite[\S 2.3]{GSZ}. They gave a 
formula for $\sum_{q\geq 0}(-1)^{q+1}q\zeta_q'(0)$. 
One can see that $\sum_{q\geq 0}(-1)^{q+1}q\zeta_q'(0)$ is a $\Q$-linear combination of
$\zeta_\Q(-m)$ and $ \zeta_\Q'(-m)$ for $m$ odd with 
 $1\leq m\leq n $ and the logarithm of a positive 
 integer.  A similar formula is given in \cite{SpreaficoProj} in low dimensional case.

 We introduce
some  notations (see Paragraph \eqref{towww} for more
details). Let $1\leq q\leq n$. For $i=0,\ldots,2n$, we
  denote by $P_i(q)$ (resp.
 $Q_i(q)$) the following integers given by   
\[
\binom{l+q-1
}{n} \binom{l-1}{l-n-1}=\sum_{i=0}^{2n}P_i(q)l^i
 \quad
\forall l\in  \N_{\geq 1}.
\]
(resp.)
\[
\binom{l+n
}{n} \binom{l-q+n}{l-q}=\sum_{i=0}^{2n}Q_i(q)l^i\quad
\forall l\in  \N_{\geq q}.
\]

Our main second result is a formula for 
$\overline{\zeta}_q'(0)$ and thus for $\zeta_q'(0)$ for any $q=1,\ldots,n$.
More precisely, we prove the following equality
(see Corollary \eqref{coroll} 1.)

 \begin{align*}
\overline{\zeta}_q'(0)=&4\frac{\binom{n}{q-1}}{(n!)^4}\zeta_\Q'(-2n+1)+
\frac{\binom{n}{q-1}}{(n!)^2(n+1-q)}\bigl(-3P_1(q)-Q_1(q)\bigr)
\zeta'_\Q(0)\\
&+\sum_{p=1}^{2n-2}\Bigl(5\widetilde{\textbf{c}}_{q,p-1}^{(n)}-\textbf{c}_{q,p-1}^{(n)}+3 \frac{\binom{n}{q-1}}{(n!)^2}P_{p+1}(q)+(-1)^p  \frac{\binom{n}{q-1}}{(n!)^2}Q_{p+1}(q)\Bigr)\zeta_\Q'(-p)\\
&+3(n+1+q)\frac{\binom{n+q-1}{n}}{q(n+1)}\log(n+1)+
(n+1-3q)\frac{\binom{n+q-1}{n}}{q(n+1)}\log(q)\\
&+2(n+1-q)\xi_q(0)
+\sum_{j=2}^{2n} (n+1-q)^j w_j\gamma_{n,q}(j-1),
\end{align*}
where \begin{align*}
\textbf{c}_{q,p}^{(n)}
=\frac{\binom{n}{q-1}}{(n!)^2}\sum_{i=p+2}^{
2n}(-1)^iQ_i(q)(n+1-q)^{i-2-p}\quad \text{and}\quad 
\widetilde{\textbf{c}}_{q,p}^{(n)}=\frac{\binom{n}{q-1}}{(n!)^2}\sum_{i=p+2}^{
2n}P_i(q)(n+1-q)^{i-2-p},
\end{align*}
with $p=0,1,\ldots,2n-2$,  $\xi_q(0)$ and $\gamma_{n,q}(j-1)$ are 
given in Proposition \eqref{prop1}. Also we give a formula for 
\[
\sum_{q=1}^n(-1)^{q+1}q\zeta_q'(0),
\]
(see Theorem \eqref{DetReg}). This theorem recovers partially the result of 
Gillet, Soul\'e and Zagier on the computation of
the holomorphic analytic torsion associated
to the Fubini-Study metric on $\p^n(\C)$, \cite{GSZ}.\\

In this article  $\zeta_\Q$ denotes
the Riemann zeta function. 

\section{The case of dimension 1}

Let $\p^1(\C)$ be the complex projective line endowed with 
the Fubini-Study metric. The spectrum of the Laplacian 
acting on $A^{(0,0)}(\p^1(\C))$ and   associated to  the Fubini-Study metric on $\p^1$ is computed and is 
 well known.
Let $\zeta_{{}_{FS}}$ be the zeta function associated to 
this spectrum. \\

 In \cite[Appendix C]{Weisberger}, this function is denoted
 $\zeta^{(0)}$.   We have 
  
  \begin{equation}\label{x1}
\zeta_{{}_{FS}}(s)=\zeta^{(0)}(s)=\sum_{l=1}^\infty \frac{2l+1}{(l(l+1))^s}.
\end{equation}

\begin{theorem}\label{xx1}
We have, for any $m\in \N$ 
\[
 \zeta_{{}_{FS}}(-m)= 
 \zeta^{(0)}(-m)=-\frac{1}{m+1}\sum_{k=m+1}^{2m+2}(-1)^k \binom{m+1}{2m+2-k}B_k.
\]
where $B_k$ is the $k$-th Bernoulli number.
\end{theorem}

By \cite[C4, p.197]{Weisberger}, we have
\begin{equation}\label{x2}
\zeta^{(0)}(s)=\frac{1}{\Gamma(s)}\sum_{k=0}^\infty 
\frac{(2s+k-2)\Gamma(s+k-1)}{\Gamma(k+1)}\zeta_\Q
(2s+k-1,2),
\end{equation}
where $\zeta_\Q(\cdot,a)$ is the generalized
 Riemann zeta function defined for $\Re(s)>1$ by
 \[
 \zeta_\Q(s,a):=\sum_{l=0}^\infty \frac{1}{(l+a)^s}\quad
\text{for}\, a\in \N_{\geq 1}.
 \]

 Recall that $z\rightarrow  \Gamma(z)$ is a 
 meromorphic  function on $\C$ with simple pole at 
 $z=-m$  for any $m\in \N$ with residue $(-1)^m/m!$. 
 Let $m\in \N$. We have
in a neighbourhood of $-m$ the following 
{\allowdisplaybreaks
 \begin{align*}
 \zeta^{(0)}(s)=&\frac{1}{\Gamma(s)}\sum_{k=0}^\infty 
\frac{(2s+k-2)\Gamma(s+k-1)}{\Gamma(k+1)} \bigl(\zeta_\Q
(2s+k-1)-1\bigr)\\
=&\frac{1}{\Gamma(s)}\sum_{k=0}^{m+1}
\frac{(2s+k-2)\Gamma(s+k-1)}{\Gamma(k+1)} \bigl(\zeta_\Q
(2s+k-1)-1\bigr)\\
&+\frac{1}{\Gamma(s)}\sum_{k=m+2}^\infty 
\frac{(2s+k-2)\Gamma(s+k-1)}{\Gamma(k+1)} \bigl(\zeta_
\Q
(2s+k-1)-1\bigr)\\
=&\frac{1}{\Gamma(s)}\sum_{k=0}^{m+1}
\frac{(2s+k-2)\Gamma(s+k-1)}{\Gamma(k+1)} \bigl(\zeta_
\Q
(2s+k-1)-1\bigr)+O(s+m)\\
=& (-1)^m m!(s+m)\sum_{k=0}^{m+1}\frac{(-2m+k-2)
(-1)^{m-k}}{\Gamma(k+1)
(m+1-k)!(s+m)}\bigl(\zeta_\Q(-2m+k-1)-1 \bigr)+O(s+m)
\\
=&m!\sum_{k=0}^{m+1}\frac{(-1)^k(-2(m+1)+k)}{k! (m+
1-k)!}
\bigl(-\frac{B_{2(m+1)-k}}{2(m+1)-k}-1 \bigr)+O(s+m)
\\
=&-\frac{1}{m+1}\sum_{k=0}^{m+1}(-1)^k \binom{m+1}{k}B_{
2(m+1)-k}+O(s+m).
 \end{align*}
 }
 Thus,
 \begin{equation}\label{x3}
 \zeta^{(0)}(-m)=-\frac{1}{m+1}\sum_{k=0}^{m+1}(-1)^k \binom{m+1}{k}B_{2(m+1)-k}=-\frac{1}{m+1}\sum_{k=m+1}^{2m+2}(-1)^k \binom{m+1}{2m+2-k}B_k.
 \end{equation}
This ends the proof of Theorem \eqref{xx1}.
\section{The general case}
We keep the same notations as in the introduction. 
For any $n$ and $q$ in $ \N_{\geq 1}$ with $n\geq q$, we  set
\begin{equation}\label{e4}
S_{n,q}(z)=\sum_{k\geq q} \alpha_{n,q}(k) z^k\quad 
\forall |z|<1,
\end{equation}
with $
\alpha_{n,q}(k):=\frac{1}{n!(q-1)!(n-q)!}\frac{\binom{k+n}{k} \binom{k-q+n}{k-q}}{k(k
+n+1-q)}$ for any $k\geq q$.

\begin{claim}\label{claim1} Let $1\leq q\leq n$.
 We have
\begin{equation}
S_{n,q}(z)=\frac{ R_{n,q}(z)}{(1-z)^{2n-1}}-
\frac{\binom{n+q-1}{
 n}}{
 q(n+1)}z^{q} \quad \forall |z|<1,
\end{equation}
where  $R_{n,q}(z)=\sum_{k=q}^{2n-1}c_{n,q,j}z^j$ is a polynomial of degree
$\leq 2n-1$ and has a zero at  $z=0$ of order $\geq q$.
\end{claim}
\begin{proof}
Let $T_{n,q}(z)=\sum_{k\geq 0} \binom{k+q+n}{k+q}
\binom{k+n}{k}z^k$ for any $|z|<1$. It is clear that $T_{n,q}$ 
converges  on $\{|z|<1\}$. 
We can check easily the following equality 
\begin{equation}\label{Texp}
\frac{d}{dz}\bigl(z^{n+2-q} \frac{d}{dz }S_{n,q}(z)
\bigr)= z^n T_{n,q}(z)\quad \forall |z|<1.
\end{equation}

Let $A(z)=\sum_{k=0}
^\infty \binom{k+q+n}{k+q}z^k$ and 
$B(z)=\sum_{k=0}^\infty
\binom{k+n}{k}z^k$. We have 
$B(z)=\frac{1}{(1-z)^{n+1}} $ and
$A(z)=\frac{1}{z^q(1-z)^{n+1}}-\frac{\sum_{k=0}^{q-1} 
\binom{k+n}{k}z^k}{z^q}$ for any $0<|z|<1$. Consider
$\alpha>1$. We have for any $0<r<1$,
\[
\int_0^{2\pi} A(\frac{e^{i\theta}}{\alpha})B(r e^{-i\theta})d\theta 
=T_{n,q}(\frac{r}{\alpha}).\]
 Note 
that $\frac{1}{(x(1-x))^{n+1}}=G_{n+1}(\frac{1}{x})+
G_{n+1}(\frac{1}{1-x})$ where $G_{n+1}(z)=\sum_{k=0}^n
a_{n+1,k}z^k$ is a polynomial 
of degree $n$ with rational coefficients. 
Since,
\[
A(\frac{e^{i\theta}}{\alpha})B(re^{-i\theta})=
\frac{\alpha^{n+1} e^{i(n+1-q)\theta}}{(\alpha-e^{i\theta})^{n+1}(e^{i\theta}-r)^{n+1}}-
\alpha^q\frac{\sum_{k=0}^{q-1} \binom{k+n}{k} 
e^{i(n+1+k-q)\theta}}{(e^{i\theta }-r)^{n+1}},
\]
Then,
\begin{align*}
\int_0^{2\pi} A(\frac{e^{i\theta}}{\alpha})B(re^{-i\theta})d\theta&=
\int_{|z|=1} \frac{\alpha^{n+1} z^{n+1-q}}{(\alpha-z)^{n+1}(z-r)^{n+1}}\frac{dz}{z}
-\alpha^q\sum_{k=0}^{q-1} \binom{k+n}{k}\int_{|z|=1} 
\frac{ z^{n+1+k-q}}{(z-r)^{n+1}}\frac{dz}{z}.
\end{align*}

It is clear that
   $\frac{1}{(\alpha-z)^{n+1}(z-r)^{n+1}}=
\frac{1}{(\alpha-r)^{2n+2}}\bigl(G_{n+1}(\frac{\alpha-r}{
z-r})+G_{n+1}(\frac{\alpha-r}{\alpha-z})  \bigr)$. This gives
\begin{align*}
\int_0^{2\pi} A(\frac{e^{i\theta}}{\alpha})B(re^{-i\theta})d\theta=&\frac{\alpha^{n+1}}{(\alpha-r)^{2n+
2}}\int_{|z|=1} 
 z^{n+1-q} \bigl(G_{n+1}(\frac{\alpha-r}{
z-r})+G_{n+1}(\frac{\alpha-r}{\alpha-z})  \bigr)\frac{dz}{z}\\
&-\alpha^q\sum_{k=n+1-q}^{n} \binom{q-1+k}{n}\int_{|z|=1} 
\frac{ z^{k}}{(z-r)^{n+1}}\frac{dz}{z}\\
=&\frac{\alpha^{n+1}}{(\alpha-r)^{2n+
2}}\int_{|z|=1} 
 z^{n+1-q} G_{n+1}(\frac{\alpha-r}{
z-r})\frac{dz}{z}-\alpha^q \binom{n+q-1}{n}\\
=& \frac{\alpha^{n+1}}{(\alpha-r)^{2n+
2}}\sum_{k=0}^{n}a_{n+1,k+1}(\alpha-r)^{k+1}
\int_{|z|=1}\frac{z^{n+1-q
}}{(z-r)^{k+1}}\frac{dz}{z}-  \alpha^q \binom{n+q-1}{n}\\
=&\frac{\alpha^{n+1}}{(\alpha-r)^{2n+
2}}\sum_{k=0}^{n+1-q}a_{n+1,k+1}(\alpha-r)^{k+1}
\binom{n+1-q}{k}r^{n+1-q-k}-\alpha^q \binom{n+q-1}{n}.
\end{align*}

It follows that
\[
T_{n,q}(\frac{r}{\alpha})=\frac{\alpha^{n+1}}{(\alpha-r)^{2n+
2}}\sum_{k=0{}}^{n+1-q}a_{n+1,k+1}(\alpha-r)^{k+1}
\binom{n+1-q}{k}r^{n+1-q-k}-\alpha^q \binom{n+q-1}{n}.
\]
We let $\alpha\rightarrow 1$. Then we obtain 
the following 
\begin{equation}\label{c2}
\begin{split}
T_{n,q}(r)&=\frac{1}{(1-r)^{2n+
1}}\sum_{k=0{}}^{n+1-q}a_{n+1,k+1}(1-r)^k
\binom{n+1-q}{k}r^{n+1-q-k}-\binom{n+q-1}{n}\\
&=\frac{P_{n,q}(r)}{(1-r)^{
 2n+1}} -\binom{n+q-1}{n} \quad \forall 0<r<1. 
\end{split}
\end{equation}
That is 
\begin{equation}\label{r1}
T_{n,q}(z)=\frac{P_{n,q}(z)}{(1-z)^{
 2n+1}} -\binom{n+q-1}{n}\quad \forall |z|<1,
\end{equation}
 where $
 P_{n,q}(z)=\sum_{k=0}^{n+1-q}a_{n+1,k+1}
 \binom{n+1-q}{k} (1-z)^k
z^{n+1-q-k}.$  
  Now using \eqref{Texp}, we obtain
 \begin{align*}
 \frac{d}{dz}\bigl(z^{n+2-q} \frac{d}{dz }S_{n,q}(z)
\bigr)&=
 z^n \frac{P_{n,q}(z)}{(1-z)^{
 2n+1}}-\binom{n+q-1}{n}z^n \\
 &=\sum_{k=0}^{n+1-q}a_{n+1,k+1}
 \binom{n+1-q}{k} \frac{z^{n+1-q-k}}{(1-z)^{2n+1-k}}-
 \binom{n+q-1}{n}z^n.
 \end{align*}
 This equality  gives
 \begin{equation}\label{r2}
 z^{n+2-q} \frac{d}{dz }S_{n,q}(z)=
 \frac{z^m Q_{n,q}(z)}{(1-z)^{2n}}-\frac{\binom{n+q-1}{
 n}}{
 (n+1)}z^{n+1}\footnote{We use the following fact: Let $0<a\leq b$ two integers with $b-a\geq 2$. Then
a primitive for $\frac{z^a}{(1-z)^b}$
 is given as a  linear combination of 
 $\frac{1}{(1-z)^i}$ with $b-a-1\leq i\leq b-1$.}
 \end{equation}
with $m$ is a nonnegative integer and $Q_{n,q}$ is a polynomial such that
$m+\deg(Q_{n,q})\leq 2n$ and $Q_{n,q}(0)\neq 0$. 
Since $S_{n,q}$ has a zero of  order $q$ at $z=0$, then
$n+2-q+q-1= n+1\leq m$. So,
\[
  \frac{d}{dz }S_{n,q}(z)=
 \frac{z^{m-n+q-2} Q_{n,q}(z)}{(1-z)^{2n}}-
\frac{\binom{n+q-1}{
 n}}{
 (n+1)}z^{q-1},
\]
and $\deg(Q_{n,q})\leq n-1$. 
We conclude that there exists a rational polynomial $R_{n,q}$ 
of degree $\leq 2n-1$ such that
\begin{equation}\label{c4}
S_{n,q}(z)=\frac{ R_{n,q}}{(1-z)^{2n-1}}-
\frac{\binom{n+q-1}{
 n}}{
 q(n+1)}z^{q}.
\end{equation}
And, we can see that the  order of $R_{n,q}$ at $z=0$ is $\geq q$.\\
\end{proof}

We set 
\[
\xi_q(s)=\frac{1}{\Gamma(s)}\int_0^\infty t^{s-1}
S_{n,q}(e^{-t})e^{-(n+1-q)t}dt \quad \forall \Re(s)\gg 1.\] 

From Claim \eqref{claim1}, we get
\begin{equation}\label{ccc1}
\xi_q(s)=\sum_{k=q}^{2n-1}c_{n,q,k}\zeta_{2n-1}
(s, k+n+1-q)-\frac{\binom{n+q-1}{
 n}}{
 q(n+1)^{s+1}},
\end{equation}
where $\zeta_{2n-1}(s,a)$ is the  multiple Hurwitz zeta
function. \\

We write 
\[
\binom{z+2n-2}{2n-2}=\frac{1}{(2n-2)!}
(z+2n-2)(z+2n-3)\cdots (z+1)=\sum_{i=0}^{2n-2}b_{2n-1
,i}z^i,
\]
where $b_{k,i}$ are rational numbers related 
to Stirling numbers.

\begin{claim}\label{12}
Let $k\in \N$. We have for any $l\in \N$
\[
\zeta_{2n-1}(-l,k+n+1-q)=\sum_{i=0}^{2n-2}b_{2n-1,i}
\sum_{p=0}^i(q-n-1-k)^{i-p} \binom{i}{p}
\frac{B_{l+p+1}(k+n+1-q)}{p+l+1},
\]
where $B_l(\cdot)$ is the $l$-th Bernoulli polynomial. 
Moreover, the residue of $\zeta_{2n-1}(s,k+n+1-q)$ at $s=l\in \{1,\ldots,2n-1
\}$ is equal to
\[
\sum_{i=l-1}^{2n-2} b_{2n-1,i}\sum_{p=0}^i(q-n-1-k)^{i
-l+1}\binom{i}{l-1}.
\]
\end{claim}

\begin{proof}
Let $\alpha$ be a positive integer. We have
\begin{equation}\label{dft}
\begin{split}
\zeta_{2n-1}(s,\alpha)&=\sum_{k=0}^\infty 
\frac{\binom{k+2n-2}{2n-2}}{(k+\alpha)^s}\\
&=\sum_{i=0}^{2n-2}b_{2n-2,i} \sum_{k=0}^\infty \frac{k^i}{(k+\alpha)^s}\\
&=\sum_{i=0}^{2n-2}b_{2n-1,i}\sum_{k=0}^\infty \frac{(k+\alpha-\alpha)^i}{(k+\alpha)^s}\\
&=\sum_{i=0}^{2n-2}b_{2n-1,i}\sum_{p=0}^i(-\alpha)^{i-
p}
\binom{i}{p}\zeta_\Q(s-p,\alpha).
\end{split}
\end{equation}
Then the claim follows from the following 
well-known formula
\[
\zeta_\Q(-l,\alpha)=-\frac{B_{l+1}
(\alpha)}{l+1},
\]
and the fact that $\zeta_\Q$ has residue $1$ at
$s=1$.
\end{proof}

For any $l\in \N$ and $0\leq p\leq i \leq 2n-2
$ we set 
\begin{equation}\label{beta}
\beta_{n,q}(l,i,p):=\sum_{k=q}^{2n-1}c_{n,q,k}(q-n-
1-k)^{i-p}B_{l+p+1}
(k+n+1-q),
\end{equation}
and for any $l\in \{1,\ldots,2n-1\}$
\begin{equation}
\gamma_{n,q}(l):=\sum_{k=q}^{2n-1}
c_{n,q,k}\sum_{i=l-1}^{2n-2} b_{2n-1,i}\sum_{p=0}^i(q-n-1-k)^{i
-l+1}\binom{i}{l-1}.
\end{equation}

\begin{proposition}\label{prop1}
The function $\xi_q$ is  holomorphic on $\Re(s)>2n-1$ and admits 
a meromorphic extension to $\C$ such that its 
poles are included into $\{1,2,\ldots,2n-1\}$. 
Moreover, the residue of $\xi_q$ at $l=1,\ldots,2n-1$
is a rational number  equal to  
$\gamma_{n,q}(l)$, and  we have
\[
\xi_q(-l)=\sum_{i=0}^{2n-2}b_{2n-1,i}
\sum_{p=0}^i \binom{i}{p}
\frac{\beta_{n,q}(l,i,p)}{p+l+1}-
\frac{\binom{n+q-1}{n}}{q}(n+1)^{l-1}\quad
\forall l\in \N.
\]
\end{proposition}
\begin{proof} The proof follows from 
Equality \eqref{ccc1}, the fact that $\zeta_\Q$ has residue $1$ at
$s=1$ and Claim \eqref{12}. 
\end{proof}

\begin{claim}\label{n1} We have for any $1\leq q\leq n$, and $\Re(s)
> n$
\begin{align*}
\overline{\zeta}_q(s)&=\frac{1}{\Gamma(s)}
\sum_{j=0}^\infty (n+1-q)^j  \frac{\Gamma(s+j-1)
(2s-2+j)}{\Gamma(j+1)} \xi_q(2s+j-1).
\end{align*}
\end{claim}
\begin{proof}
Obviously, we have
\begin{equation}\label{e2}
\overline{\zeta}_q(s)=\sum_{k\geq q+1}\alpha_{n,q}(k)\Bigl(
\frac{1}{k^s(k+n+1-q)^{s-1}}+\frac{1}{k^{s-1}(k+n+1-q)^{s
}} \Bigr).
\end{equation}

Recall that $\Gamma(s)l^{-s}=\int_0^\infty t^{s-1}e^{-lt}dt$.  Then,
\begin{align*}
\overline{\zeta}_q(s)&=\frac{1}{\Gamma(s-1)\Gamma(s)}\int_0^\infty \int_0^\infty
(uv)^{s-2}(u+v) \sum_{k\geq q+1} \alpha_{n,q}(k) \exp\bigl
(-k(u+v)-(n+1-q)v\bigr)du
dv\\
&=\frac{1}{\Gamma(s-1)\Gamma(s)}\int_0^\infty\int_0^1 (\theta (1-\theta))^{s-2}
t^{2s-2}\sum_{k\geq q+1} \alpha_{n,q}(k)\exp(-k t) \exp(-
(n+1-q)(1-\theta)t)d\theta dt\\
&=\frac{1}{\Gamma(s-1)\Gamma(s)}\int_0^\infty t^{2s-2}
 \sum_{k\geq q+1} \alpha_{n,q}(k)\exp(-k t)\exp(-(n+1-q)t) 
 \int_0^1(\theta (1-\theta))^{s-1}
\exp(
(n+1-q)\theta t) d\theta dt\\
&=\frac{1}{\Gamma(s-1)\Gamma(s)}\int_0^\infty t^{2s-2}
S_{n,q}(e^{-t})\exp(-(n+1-q)t) 
 \int_0^1(\theta (1-\theta))^{s-1}
\exp(
(n+1-q)\theta t) d\theta dt.
\end{align*}

We set $u:=\theta t $ and  $v:=(1-\theta)t$. We have
\begin{align}
  \int_0^1(\theta (1-\theta))^{s-2}
\exp(
(n+1-q)\theta t) d\theta&= \sum_{j=0}^\infty \frac{(n+1-q)^j}{
j!} \bigl(\int_0^1\theta^{s+j-2}(1-\theta)^{s-2}d\theta\bigr) t^j\\
&=\sum_{j=0}^\infty (n+1-q)^j \frac{\Gamma(s-1)\Gamma(s+j-1)}{\Gamma(j+1)\Gamma(2s-2+j)}t^j.
\end{align}

Thus,
\begin{equation}\label{z1}
\overline{\zeta}_q(s)=\frac{1}{\Gamma(s)}
\sum_{j=0}^\infty (n+1-q)^j  \frac{\Gamma(s+j-1)}{\Gamma(j+1)\Gamma(2s-2+j)} \int_0^\infty t^{2s+j-2} S_{n,q}(e^{-t}) \exp(-(n+1-q)t)dt.
\end{equation}

Therefore,
\[
\overline{\zeta}_q(s)=\frac{1}{\Gamma(s)}
\sum_{j=0}^\infty (n+1-q)^j  \frac{\Gamma(s+j-1)
(2s-2+j)}{\Gamma(j+1)} \xi_q(2s+j-1).
\]
\end{proof}

\subsection{The values of $\zeta_q$ at non-positive
integers}
 
In this paragraph, we study the zeta function $\zeta_q
$ at non-positive integers, also  the value of its 
derivative at $s=0$. The following theorem gives  formulas for 
the values of $\overline{\zeta}_q$ and thus of
$\zeta_q$ at non-positive integers. 
\begin{theorem}\label{ValuesZeta}
For any $1\leq q \leq n$ and  any $m\in \N$,
\begin{align*}
\overline{\zeta}_q(-m)=&\sum_{j=0}^{2m+1}
(-1)^{j+1}\frac{(n+1-q)^j}{m}\binom{m+1}{j} 
\xi_q(-2m+j-1) \\
&+(-1)^m m!\sum_{j=2m+2}^{2n+2m} (n+1-q)^j  \frac{\Gamma(-m
+j-1)
(-2m-2+j)}{\Gamma(j+1)} \gamma_{n,q}(-2m+j-1),
\end{align*}
Recall that the value of $\xi_q(-2m+j-1)$ is given
in Proposition \eqref{prop1}. In particular,
\[
\overline{\zeta}_q(-m)\in \Q.
\]
\end{theorem}

\begin{proof} Let $m\in \N$.
By Claim \eqref{n1}  we have for any 
$s$ in a small open neighborhood of $-m$ the 
following 
\begin{equation}\label{qq1}
\begin{split}
\overline{\zeta}_q(s)=&\frac{1}{\Gamma(s)}
\sum_{j=0}^{2m+1} (n+1-q)^j  \frac{\Gamma(s+j-1)
(2s-2+j)}{\Gamma(j+1)} \xi_q(2s+j-1)\\
&+\frac{1}{\Gamma(s)}
\sum_{j=2m+2}^{2n+2m} (n+1-q)^j  \frac{\Gamma(s+j-1)
(2s-2+j)}{\Gamma(j+1)} \xi_q(2s+j-1)\\
&+\frac{1}{\Gamma(s)}
\sum_{j=2n+2m+1}^\infty (n+1-q)^j  \frac{\Gamma(s+j-1)
(2s-2+j)}{\Gamma(j+1)} \xi_q(2s+j-1).
\end{split}
\end{equation}

First, we can show by a classical
argument that \[
\sum_{j=2n+2m+1}^\infty (n+1-q)^j  \frac{\Gamma(s+j-1)
(2s-2+j)}{\Gamma(j+1)} \xi_q(2s+j-1)
\]
converges normally  on  $\Re(s)>-m$. Thus, the 
third sum in \eqref{qq1} vanishes at $s=-m$.

The first sum is equal at $s=-m$ to $\sum_{j=0}^{2m+1}
(-1)^{j+1}\frac{(n+1-q)^j}{m}\binom{m+1}{j} 
\xi_q(-2m+j-1) $. 
For $2m+2\leq j\leq 2n+2m$, the function $\xi_q(2s+j-1)$ may have 
a pole at $s=-m$. In this case, we know (see Proposition 
\eqref{prop1}) that
\[
\xi_q(2s+j-1)=\frac{\gamma_{n,q}(-2m+j-1)}{s+m}+o(1),
\quad 2m+2\leq j\leq 2n+2m.
\]
as  $s\rightarrow -m$.\\

Gathering all these computations, we obtain that
\begin{align*}
\overline{\zeta}_q(-m)=&\sum_{j=0}^{2m+1}
(-1)^{j+1}\frac{(n+1-q)^j}{m}\binom{m+1}{j} 
\xi_q(-2m+j-1) \\
&+(-1)^m m!\sum_{j=2m+2}^{2n+2m} (n+1-q)^j  \frac{\Gamma(-m
+j-1)
(-2m-2+j)}{\Gamma(j+1)} \gamma_{n,q}(-2m+j-1)
\end{align*} 
Recall that, by Proposition \eqref{prop1}, we have
\begin{align*}
\xi_q(-2m+j-1)=\sum_{i=0}^{2n-2}b_{2n-1,i}
\sum_{p=0}^i \binom{i}{p}
\frac{\beta_{n,q}(2m-j+1,i,p)}{p+2m-j+2}-
\frac{\binom{n+q-1}{n}}{q}(n+1)^{2m-j}.
\end{align*}
We conclude that $\overline{\zeta}_q(-m)\in \Q.$

\end{proof}

Let $\Re(s)>2n-1$ and $1\leq q\leq n$, we set
\[
\eta_q(s)=\frac{1}{\Gamma(s)}\int_0^\infty 
t^{s-1}S_{n,q}(e^{-t})dt.
\]
Then,
\begin{equation}\label{dff}
\eta_q(s)=\sum_{k=q}^{2n-1}c_{n,q,k}\zeta_{2n-1}(s,
k)-\frac{\binom{n+q-1}{
 n}}{
 (n+1) q^{s+1}}.
\end{equation}
By a similar argument as in Proposition
\eqref{prop1} we can show that $\eta_q$ is holomorphic
on $\Re(s)>2n-1$ and admits a meromorphic extension
to $\C$ with poles in $\{1,\ldots,2n-1\}$. \\

The following theorem provides the first formula
for $\overline{\zeta}_q'(0)$. This formula will
be simplified in the sequel.
\begin{theorem}\label{rtu}
For any $1\leq q\leq n$, we have 

\begin{enumerate}[i.]
\item 
\begin{align}
\xi_q'(-1)&=\sum_{k=q}^{2n-1}c_{n,q,k}\sum_{i=0}^{2n-2}b_{2n-1,i}\sum_{p=0}^i(q-n-1-k)^{i-p}
\binom{i}{p}\zeta_\Q'(-p-1)+\frac{\binom{n+q-1}{
 n}}{
 q}\log(n+1),\\
\xi_q'(0)&=\sum_{k=q}^{2n-1}c_{n,q,k}\sum_{i=0}^{2n-2}b_{2n-1,i}\sum_{p=0}^i(q-n-1-k)^{i-p}
\binom{i}{p}\zeta_\Q'(-p)+\frac{\binom{n+q-1}{
 n}}{
 q(n+1)}\log(n+1),\\
\eta_q'(-1)&=\sum_{k=q}^{2n-1}c_{n,q,k}\sum_{i=0}^{2n-2}b_{2n-1,i}\sum_{p=0}^i(-k)^{i-p}
\binom{i}{p}\zeta_\Q'(-p-1)+\frac{\binom{n+q-1}{
 n}}{
 (n+1)}\log(q),\\
\eta_q'(0)&=\sum_{k=q}^{2n-1}c_{n,q,k}\sum_{i=0}^{2n-2}b_{2n-1,i}\sum_{p=0}^i(-k)^{i-p}
\binom{i}{p}\zeta_\Q'(-p)+\frac{\binom{n+q-1}{
 n}}{
 (n+1)q}\log(q).
\end{align}
\item 
We have the following expression for $\overline{\zeta}
_q'(0)$
\begin{align*}
\overline{\zeta}_q'(0)=&6 \xi'_q(-1)-3(n+1-q)\xi_q'(0)-2\eta_q'(-1)+(n+1-q)
\eta'_q(0)+2(n+1-q)\xi_q(0)\\
&+\sum_{j=2}^{2n} (n+1-q)^j w_j\gamma_{n,q}(j-1).
\end{align*}
\end{enumerate}

where $w_j$ is a rational number equal to $\frac{d}{ds}\Bigl(\frac{\Gamma(s+j-1)(2s-2+j)}{\Gamma(j+1)\Gamma(s+1)} \Bigr)_{|_{s=0}}$ for
$j=2,\ldots,2n$.
\end{theorem}
\begin{proof}

\begin{enumerate}[i.]
\item  These  expressions follow directly from
\eqref{ccc1}, \eqref{dft} and \eqref{dff}.

\item 
From Claim \eqref{n1} and using similar arguments as before, we have 
for any $s$ in a small
 open neighborhood of $0$, the following equality
\begin{align*}
\overline{\zeta}_q(s)=&2\xi_q(2s-1)+(n+1-q)(2s-1)\xi_q
(2s)+\sum_{j=2}^{2n} (n+1-q)^j\frac{\Gamma(s+j-1)(2s-2+j)}{\Gamma(j+1)}
\Bigl(\frac{\gamma_{n,q}(j-1)}{\Gamma(s+1)}\\
&+s o(1)\Bigr
)
+\frac{1}{\Gamma(s)}\sum_{j=2n+1}^\infty (n+1-q)^j 
 \frac{\Gamma(s+j-1)
(2s-2+j)}{\Gamma(j+1)} \xi_q(2s+j-1).
\end{align*}
Observe that $\sum_{j=2n+1}^\infty (n+1-q)^j 
 \frac{\Gamma(s+j-1)
(2s-2+j)}{\Gamma(j+1)} \xi_q(2s+j-1)$ converges 
normally in an open neighbourhood of $s=0$. 
Then, we obtain easily the following
\begin{align*}
\overline{\zeta}_q'(0)=&4\xi_q'(-1)+2(n+1-q)(\xi_q(0)-
\xi_q'(0))+\sum_{j=2}^{2n} (n+1-q)^j w_j\gamma_{n,q}(j-1)
\\
&+\sum_{j=2n+1}^\infty (n+1-q)^j \frac{j-2}{j(j-1)}
\xi_q(j-1),
\end{align*}
where $w_j=\frac{d}{ds}\Bigl(\frac{\Gamma(s+j-1)(2s-2+j)}{\Gamma(j+1)\Gamma(s+1)} \Bigr)_{s=0}$ which
is clearly a rational number.\\

Let us evaluate the sum $\sum_{j=2n+1}^\infty (n+1-q)^j
 \frac{j-2}{j(j-1)}
\xi_q(j-1)$. Notice that 
 $\sum_{j=2n+1}^\infty \frac{j-2}{j(j-1)\Gamma(j-1)}t^{j-2}= (t^{-1}-2t^{-2})(e^t-\sum_{j=0}^{2n-1} 
 \frac{t^j}{j!})+2\frac{t^{2n-2}}{(2n)!}$ for any 
 $t>0$. This gives
\begin{equation}\label{rrrr1}
\begin{split}
&\sum_{j=2n+1}^\infty (n+1-q)^j\frac{j-2}{j(j-1)}
\xi_q(j-1)=
\int_0^\infty \sum_{j=2n+1}^\infty (n+1-q)^{j} 
\frac{j-2}{j(j-1)\Gamma(j-1)} t^{j-2} S_{n,q}(
e^{-t})e^{-(n+1-q)t}dt\\
=&\int_0^\infty\Bigl(\bigl((n+1-q)t^{-1}-2t^{-2}\bigr)\bigl(e^{(n+1-q)t}-\sum_{j=0}^{2n-1}(n+1-q)^j\frac{t^j}{j!}\bigr)+2(n+1-q)^{2n}t
^{2n-2}\Bigr) S_{n,q}(
e^{-t})e^{-(n+1-q)t}dt.
\end{split}
\end{equation}

We introduce the following function  $\Omega$ 
given on $\Re(s)\gg 1$ by  
\begin{equation}\label{gammag}
\begin{split}
\Omega(s):=&(n+1-q)\Gamma(s)\eta_q(s)-2 \Gamma(s-1)\eta_q(s-1)-\sum_{j=0}^{2n-1} \frac{(n+1-q)^j}{j!}\Bigl[(
n+1-q)\Gamma(s+j)
\xi_q(s+j)\\
&-2\Gamma(s+j-1)\xi_q(s+j-1)  \Bigr]
+2(n+1-q)^{2n+1}\Gamma(s+2n-2)\xi_q(s+2n-2)\\
&-
4(n+1-q)^{2n}\Gamma(s+2n-3)\xi_q(s+2n-3).
\end{split}
\end{equation}
We can see easily that $\Omega$ is holomorphic 
on $\Re(s)>2n$ and 
\begin{equation}\label{formula1}
\Omega(s)=\int_0^\infty\Bigl(\bigl((n+1-q)t^{s-1}-2t^{s-2}\bigr)\bigl(e^{(n+1-q)t}-\sum_{j=0}^{2n-1}(n+1-q)^j\frac{t^j}{j!}\bigr)+2(n+1-q)^{2n}t
^{s+2n-2}\Bigr) S_{n,q}(
e^{-t})e^{-(n+1-q)t}dt.
\end{equation}
In particular $\Omega$ admits a meromorphic 
extension to $\C$ which is moreover holomorhic at
$s=0$ (this follows from  \eqref{rrrr1}). \\

From \eqref{gammag} we have
\begin{align*}
\Omega(s)=&(n+1-q)\Gamma(s+1)\eta_q'(0)-2 \Gamma(s+1)\eta_q'(-1)+(n+1-q)\Gamma(s)\eta_q(0)-2\Gamma(s-1)
\eta_q(-1)+h_0(s)\\
&-(n+1-q)\Gamma(s)\xi_q(0)-(n+1-q)\Gamma(s+1)\xi_q'(0
) +2\Gamma(s-1)\xi_q(-1)+2\Gamma(s+1)\xi_q'(-1)+h_0(s)\\
&-\sum_{j=1}^{2n-1} \frac{(n+1-q)^j}{j!}\Bigl[(
n+1-q)\Gamma(s+j)\frac{\gamma_{n,q}(j)}{s}-2\Gamma(s+j-1)\frac{\gamma_{n,q}(j-1)}{s}  \Bigr]+h_2(s)\\
&+2(n+1-q)^{2n+1}\Gamma(s+2n-2)\frac{\gamma_{n,q}(2n-2)}{s}-
4(n+1-q)^{2n}\Gamma(s+2n-3)\frac{\gamma_{n,q}(2n-3)}{s}+h_3(s
).
\end{align*}
where $h_0$, $h_1$, $h_2$  and $h_3$ 
are smooth functions
in a open neighborhood of $s=0$, and $h(0)=h_1(0)=
h_2(0)=h_3(0)=0$.\\

Thus,
\[
\Omega(s)= (n+1-q)\eta_q'(0)-2 
\eta_q'(-1)-(n+1-q)\xi_q'(0)+2\xi_q'(-1)+
\frac{w}{s}+h(s),
\]
where $w$ is a rational number and $h$ is a smooth 
function in a open neighborhood of $s=0$ such that
$h(0)=0$.  We claim that $w=0$. Indeed, the formula \eqref{formula1} and  Equation
\eqref{rrrr1} show  that the meromorphic extension 
of  $\Omega$ is in fact holomorphic at $s=0$, that is $w=0$.

Then,
\[
\sum_{j=2n+1}^\infty (n+1-q)^j\frac{j-2}{j(j-1)}
\xi_q(j-1)=\Omega(0)=(n+1-q)\eta_q'(0)-2 
\eta_q'(-1)-(n+1-q)\xi_q'(0)+2\xi_q'(-1).
\]
We conclude that,
\begin{align*}
\overline{\zeta}_q'(0)=&4\xi_q'(-1)+2(n+1-q)(\xi_q(0)-
\xi_q'(0))+\sum_{j=2}^{2n} (n+1-q)^j w_j\gamma_{n,q}(j-1)
\\
&+(n+1-q)\eta_q'(0)-2 
\eta_q'(-1)-(n+1-q)\xi_q'(0)+2\xi_q'(-1)\\
=&6 \xi'_q(-1)-3(n+1-q)\xi_q'(0)-2\eta_q'(-1)+(n+1-q)
\eta'_q(0)+2(n+1-q)\xi_q(0)\\
&+\sum_{j=2}^{2n} (n+1-q)^j w_j\gamma_{n,q}(j-1).
\end{align*}

\end{enumerate}

This ends the proof of the theorem.
\end{proof}

\subsection{Toward a formula for the regularized
determinant}\label{towww}

In this paragraph, we propose an alternative approach
for the computation of $\sum_{q\geq 0}(-1)^{
q+1}q\zeta_q'(0)$. We will prove that $\overline{\zeta}_q'(0)$, and hence 
$\zeta_q'(0)$, is given  in terms of the derivatives of
$\zeta_\Q$ at non-positive integers.\\

 Let $1\leq q\leq n$. There exist integers denoted  $P_i(q)\in \Z$ for 
 $i=0,\ldots,2n$ such that  
\begin{equation}\label{pq1}
\binom{l+q-1
}{n} \binom{l-1}{l-n-1}=\sum_{i=0}^{2n}P_i(q)l^i\quad
\forall l\in  \N_{\geq 1}.
\end{equation}

$P_i$ can be seen as a polynomial in the variable $q$. We can prove that its  degree is $\leq 2n-i$. In 
fact, it 
suffices to note that $P_i(q)$ is expressed in terms
of the zeros of the polynomial $\binom{l+q-1
}{n} \binom{l-1}{l-n-1}$ with $l$ as a variable. 

Similarly, we define $Q_i(q)\in \Z$ for $i=0,\ldots,
2n$ verifying 
\begin{equation}\label{pq2}
\binom{l+n
}{n} \binom{l-q+n}{l-q}=\sum_{i=0}^{2n}Q_i(q)l^i\quad
l\in  \N_{\geq q}.
\end{equation}
Also, $Q_i$ is a polynomial of degree $\leq 2n-i.$\\

We set $\widetilde{\alpha}_{n,q}(l):=
\alpha_{n,q}(l-n-1+q)$ for any $l\geq n+1-q$ (see 
\eqref{e4} for the definition of $\alpha_{n,q}$). 

From Claim \eqref{claim1}, we obtain the following 
\begin{align*}
\alpha_{n,q}(l)=&\sum_{k=q}^{l} c_{n,q,k}\binom{
l-k+2n-2}{2n-2}\\
=& \sum_{k=q}^{l} c_{n,q,k} \sum_{i=0}^{2n-2}
b_{2n-1,i}(l-k)^i \quad \forall l\geq 2n-1.
\end{align*}
Which is equivalent to 
\[
\alpha_{n,q}(l)=\sum_{k=q}^{2n-2} c_{n,q,k} \sum_{
i=0}^{2n-2}
b_{2n-1,i}\sum_{p=0}^i(-k)^{i-p} \binom{i}{i-p} l^p
\quad l\geq 2n-1.
\]
In a similar way, we obtain
\[
\widetilde{\alpha}_{n,q}(l)= \sum_{k=q}^{2n-2} c_{n,q,k} \sum_{i=0}^{2n-2}
b_{2n-1,i}\sum_{p=0}^i(q-n-1-k)^{i-p} \binom{i}{i-p} l^p \quad \forall l\geq 2n-1.
\]
We denote by $\textbf{c}_{q,i}^{(n)}$ (resp. 
$\widetilde{\textbf{c}}_{q,i}^{(n)}$) the 
$i$-th coefficient of $\alpha_{n,q}$ (resp. 
$\widetilde
{\alpha}_{n,q}$). That is,
 \begin{equation}
 \widetilde{\alpha}_{n,q}(l)=\sum_{i=0}^{2n-2} 
 \textbf{c}_{q,i}^{(n)} l^i\quad\text{and}\quad \widetilde{
 \alpha}_{n,q}(l)=\sum_{i=0}^{2n-2} 
 \widetilde{\textbf{c}}_{q,i}^{(n)} l^i,\quad \forall l
 \geq 2n-1.
\end{equation}

\begin{claim}\label{zetat} Let $1\leq q\leq n$. We have,
\begin{enumerate}
\item 
\begin{align}
\xi_q'(-1)&=\sum_{p=0}^{2n-2}\widetilde{\textbf{c}}_{q,p}^{(n)}\zeta_\Q'(-p-1)+\frac{\binom{n+q-1}{
 n}}{
 q}\log(n+1),\\
\xi_q'(0)&=\sum_{p=0}^{2n-2}\widetilde{\textbf{c}}_{q,p}^{(n)}\zeta_\Q'(-p)+\frac{\binom{n+q-1}{
 n}}{
 q(n+1)}\log(n+1),\\
\eta_q'(-1)&=\sum_{p=0}^{2n-2}\textbf{c}_{q,p}^{(n)}\zeta_\Q'(-p-1)+\frac{\binom{n+q-1}{
 n}}{
 (n+1)}\log(q),\\
\eta_q'(0)&=\sum_{p=0}^{2n-2}\textbf{c}_{q,p}^{(n)}\zeta_\Q'(-p)+\frac{\binom{n+q-1}{
 n}}{
 (n+1)q}\log(q).
\end{align}
\item 

\begin{align*}
\overline{\zeta}_q'(0)=&\bigl(6 \widetilde{\textbf{c}}_{q,2n-2
}^{(n)} -2\textbf{c}_{q,2n-2}^{(n)}\bigr)\zeta_\Q'(-2n+1)+(n+1-q)
\bigl(-3\widetilde{\textbf{c}}_{q,0}^{(n)}+\textbf{c}_{q,0}^{(n)}
\bigr)
\zeta'_\Q(0)\\
&+\sum_{p=1}^{2n-2}\bigl(6\widetilde{\textbf{c}}_{q,p-1}^{(n)}-3(n+1-q)\widetilde{\textbf{c}}_{q,p}^{(n)}-2\textbf{c}_{q,p-1}^{(n)}+(n+1-q)\textbf{c}_{q,p}^{(n)}  \bigr)\zeta_\Q'(-p)\\
&+3(n+1+q)\frac{\binom{n+q-1}{n}}{q(n+1)}\log(n+1)+
(n+1-3q)\frac{\binom{n+q-1}{n}}{q(n+1)}\log(q)\\
&+2(n+1-q)\xi_q(0)
+\sum_{j=2}^{2n} (n+1-q)^j w_j\gamma_{n,q}(j-1).
\end{align*}

\end{enumerate}
Recall that $\xi_q(0)$ and $\gamma_{n,q}(j-1)$ are
given explicitly  in Proposition \eqref{prop1}.
\end{claim}
\begin{proof}
\begin{enumerate}
\item 
By definition, 
\[
\widetilde{\textbf{c}}_{q,p}^{(n)}=\sum_{k=q}^{2n-1}c_{n,q,k}\sum_{i\geq p}^{2n-2}b_{2n-1,i}(q-n-1-k)^{i-p}
\binom{i}{p}.
\]
and 
\[
\textbf{c}_{q,p}^{(n)}=\sum_{k=q}^{2n-1}c_{n,q,k}\sum_{i\geq p}^{2n-2}b_{2n-1,i}(-k)^{i-p}
\binom{i}{p}.
\]
Then $1.$ follows from Theorem \eqref{rtu}.
\item The expression of $\overline{\zeta}_q'(0)$ 
follows from Theorem \eqref{rtu} and $1.$
\end{enumerate}
\end{proof}

The following proposition gives an explicit
 formula for 
the rational numbers $\textbf{c}_{q,p}^{(n)}$ and  $\widetilde{\textbf{c}}_{q,p}^{(n)}$ for $p=0,\ldots,
2n-2$.
\begin{proposition}\label{pop2}
For any $n\geq 1$ and $1\leq q\leq n$, we have
\begin{enumerate}
\item 
\begin{align*}
\textbf{c}_{q,j}^{(n)}=\frac{\binom{n}{q-1}}{(n!)^2}\sum_{i=j+2}^{
2n}(-1)^iQ_i(q)(n+1-q)^{i-2-j},
\end{align*}
and
\begin{align*}
\widetilde{\textbf{c}}_{q,j}^{(n)}
=\frac{\binom{n}{q-1}}{(n!)^2}\sum_{i=j+2}^{
2n}P_i(q)(n+1-q)^{i-2-j},
\end{align*}
for any $j=0,1,\ldots,2n-2$. 
\item  
\[
\sum_{q=1}^{n} (-1)^q\textbf{c}_{q,j}^{(n)}=0
\quad\text{and}\quad \sum_{q=1}^{n} (-1)^q\widetilde{\textbf{c}}_{q,j}^{(n)}=0\quad
\forall j\geq n.
\]
\end{enumerate}
\end{proposition}
\begin{proof}

Let $l\gg 1$, we have
\begin{align*}
\alpha_{n,q}(l)=&\frac{\binom{l+n}{l} \binom{l-q+n}{l-q}}{n!(q-1)!(n-q)!l(l
+n+1-q)}\\
=&
\frac{l^{2n-2}}{n!(n-1)!(n-q)!}\sum_{i=0}^{2n} 
Q_{2n-i}(q) \frac{1}{l^i}\sum_{i=0}^\infty 
\frac{(-1)^i(n
+1-q)^i}{l^i}\\
&=\frac{l^{2n-2}}{n!(n-1)!(n-q)!} \sum_{i=0}^\infty 
\sum_{j=0}^i Q_{2n-i}(q)(-1)^i(n+1-q)^{i-j}\frac{1}{l^i},
\end{align*}
and

\begin{align*}
\widetilde{\alpha}_{n,q}(l)&=\frac{\binom{l-1+q}{l-n-1+q}\binom{l-1}{l-n-1}}{n!(n-1)!(n-q)!l(l-n-1+q)}\\
&=
\frac{l^{2n-2}}{n!(n-1)!(n-q)!}\sum_{i=0}^{2n} 
P_{2n-i}(q) \frac{1}{l^i}\sum_{i=0}^\infty \frac{(n
+1-q)^i}{l^i}\\
&=\frac{l^{2n-2}}{n!(n-1)!(n-q)!} \sum_{i=0}^\infty 
\sum_{j=0}^i P_{2n-i}(q)(n+1-q)^{i-j}\frac{1}{l^i}.
\end{align*}

For any $j=0,1,\ldots,2n-2$, we deduce that

\begin{align*}
\textbf{c}_{q,j}^{(n)}&=\frac{1}{n!(q-1)!(n-q)!}\sum_{i=0}^{
2n-2-j}Q_{2n-i}(q)(-1)^i(n+1-q)^{2n-2-j-i}\\
&=\frac{\binom{n}{q-1}}{(n!)^2}\sum_{i=j+2}^{
2n}(-1)^iQ_i(q)(n+1-q)^{i-2-j}.
\end{align*}
and,
\begin{align*}
\widetilde{\textbf{c}}_{q,j}^{(n)}&=\frac{1}{n!(q-1)!(n-q)!}\sum_{i=0}^{
2n-2-j}P_{2n-i}(q)(n+1-q)^{2n-2-j-i}\\
&=\frac{\binom{n}{q-1}}{(n!)^2}\sum_{i=j+2}^{
2n}P_i(q)(n+1-q)^{i-2-j}.
\end{align*}
In particular, $\textbf{c}_{q,j}^{(n)}=0$ and $\widetilde{\textbf{c}}_{q,j}^{(n)}=0$ for $j\geq 2n-1$.
 If $j\geq n$, we have $
 \sum_{i=j+2}^{
2n}(-1)^iQ_i(q)(n+1-q)^{i-2-j}$ (resp. $\sum_{i=j+2}^{
2n}P_i(q)(n+1-q)^{i-2-j}$), as a polynomial with 
$q$ as an indeteminate 
has a
degree $\leq 2n-2-j\leq n-2$. Then by a classical argument of analysis  we get 
\[
\sum_{q=1}^{n} (-1)^q\textbf{c}_{q,j}^{(n)}=0
\quad\text{and}\quad \sum_{q=1}^{n} (-1)^q\widetilde{\textbf{c}}_{q,j}^{(n)}=0\quad
\forall j\geq n.
\]
\end{proof}

\begin{corollary}\label{coroll}
Let $1\leq q\leq n$, we have
\begin{enumerate}
\item \begin{align*}
\overline{\zeta}_q'(0)=&4\frac{\binom{n}{q-1}}{(n!)^4}\zeta_\Q'(-2n+1)+
\frac{\binom{n}{q-1}}{(n!)^2(n+1-q)}\bigl(-3P_1(q)-Q_1(q)\bigr)
\zeta'_\Q(0)\\
&+\sum_{p=1}^{2n-2}\Bigl(5\widetilde{\textbf{c}}_{q,p-1}^{(n)}-\textbf{c}_{q,p-1}^{(n)}+3 \frac{\binom{n}{q-1}}{(n!)^2}P_{p+1}(q)+(-1)^p  \frac{\binom{n}{q-1}}{(n!)^2}Q_{p+1}(q)\Bigr)\zeta_\Q'(-p)\\
&+3(n+1+q)\frac{\binom{n+q-1}{n}}{q(n+1)}\log(n+1)+
(n+1-3q)\frac{\binom{n+q-1}{n}}{q(n+1)}\log(q)\\
&+2(n+1-q)\xi_q(0)
+\sum_{j=2}^{2n} (n+1-q)^j w_j\gamma_{n,q}(j-1).
\end{align*}
\item 
\begin{align*}
\sum_{q=1}^n(-1)^{q+1}\overline{\zeta}_q'(0)=&
\sum_{p=1}^n \bigl(5 \widetilde{\textbf{c}}_{p-1}^{(n)}- \textbf{c}_{p-1}^{(n)}+\textbf{d}_p
 \bigr)\zeta_\Q'(-p)-\textbf{d}_0 \zeta_\Q'(0)+
 \textbf{e}_0.
\end{align*}
with $\textbf{c}_{p}^{(n)}=\sum_{q=1}^n(-1)^{q+1} \textbf{c}_{q,p}^{(n)}$, $\widetilde{\textbf{c}}_{p}^{(n)}=\sum_{q=1}^n(-1)^{q+1} \widetilde{\textbf{c}}_{q,p}^{(n)}$, $\textbf{d}_p=\sum_{q=1}^n(-1)^q \frac{\binom{n}{q-1}}{(n!)^2}\bigl(3P_{p+1}(q)+(-1)^p Q_{p+1
}(q)\bigr)$ for $p=0,\ldots,n$ and  $\textbf{e}_0=
\sum_{q=1}^n2(-1)^q(n+1-q)\xi_q(0)
+\sum_{q=1}^n\sum_{j=2}^{2n} (-1)^q(n+1-q)^j w_j\gamma_{n,q}(j-1)+\sum_{q=1}^n3(-1)^q(n+1+q)\frac{\binom{n+q-1}{n}}{q(n+1)}\log(n+1)+\sum_{q=1}^n
(-1)^q(n+1-3q)\frac{\binom{n+q-1}{n}}{q(n+1)}\log(q)$.
\end{enumerate}
\end{corollary}
\begin{proof}
\begin{enumerate}
\item 
By using Proposition \eqref{pop2}, we have for 
any $p=1,\ldots,2n-2$
\begin{align*}
6\widetilde{\textbf{c}}_{q,p-1}^{(n)}-3(n+1-q)\widetilde{\textbf{c}}_{q,p}^{(n)}-2\textbf{c}_{q,p-1}^{(n)}+(n+1-q)\textbf{c}_{q,p}^{(n)}=&5\widetilde{\textbf{c}}_{q,p-1}^{(n)}+3 \frac{\binom{n}{q-1}}{(n!)^2}P_{p+1}(q)\\
&-\textbf{c}_{q,p-1}^{(n)}+(-1)^p  \frac{\binom{n}{q-1}}{(n!)^2}Q_{p+1}(q).
\end{align*}
We can deduce also that,
\[
\widetilde{\textbf{c}}_{q,2n-2
}^{(n)}=\frac{\binom{n}{q-1}}{(n!)^2}P_{2n}(q)\quad 
\text{and}\quad \textbf{c}_{q,2n-2}^{(n)}=
\frac{\binom{n}{q-1}}{(n!)^2}Q_{2n}(q),
\]
But, $P_{2n}(q)=Q_{2n}(q)=\frac{1}{(n!)^2}$ which
follows from \eqref{pq1} and \eqref{pq2}.  Also, note 
that 
\[
\sum_{i=0}^{
2n}(-1)^iQ_i(q)(n+1-q)^i=0,\, \quad
\sum_{i=0}^{
2n}P_i(q)(n+1-q)^i=0,\, Q_0(q)=0\, \text{and}\, P_0(q)
=0,
\]
because  $n+1-q$ and $0$  are zeros for the polynomials
in \eqref{pq1} and \eqref{pq2}.

Then 1. follows from 2. of Claim \eqref{zetat}.
\item By combining 2. of Proposition \eqref{pop2} and
$1.$ we conclude the proof of $2.$

\end{enumerate}
\end{proof}

\begin{theorem}\label{DetReg}
We keep the same notations as in Corollary \eqref{coroll}. We have,
\[
\sum_{q=1}^n(-1)^{q+1}q\zeta_q'(0)=\sum_{p=1}^n \bigl(5 \widetilde{\textbf{c}}_{p-1}^{(n)}- \textbf{c}_{p-1}^{(n)}+\textbf{d}_p
 \bigr)\zeta_\Q'(-p)-\textbf{d}_0 \zeta_\Q'(0)+\textbf{e}_0.
\]
\end{theorem}
\begin{proof}

By  \eqref{fin1}, we have
\[
\sum_{q=0}^n(-1)^{q+1}q\zeta_q(s)=\sum_{q=1}^n(-1)^
{q+1} \overline{\zeta}_q(s).
\]
Then the theorem follows from Corollary 
\eqref{coroll}.

\end{proof}

\bibliographystyle{plain}

\bibliography{biblio}

\end{document}